\documentclass[graybox]{svmult}

\usepackage{mathptmx}       
\usepackage{helvet}         
\usepackage{courier}        
\usepackage{type1cm}        
\usepackage{makeidx}         
\usepackage{graphicx}        
\usepackage{multicol}        
\usepackage[bottom]{footmisc}

\usepackage{amsmath}
\usepackage{amssymb}

\makeindex             

\newcommand\dom{\textbf{Dom}}
\newcommand\tuple{\vec}

\newcommand {\parts}{\mathcal P}

\newcommand \all{\textit{All}}
\newcommand \nonempty{\textit{NE}}

\newcommand{\FO}{\textbf{FO}}
\newcommand{\DD}{\mathcal D}
\newcommand{\Dep}{\mathbf D}

\newcommand{\model}{\mathfrak M}
\newcommand{\startteam}{\{\emptyset\}}

\begin{document}

\title*{On Strongly First-Order Dependencies}
\author{Pietro Galliani}
\institute{Pietro Galliani \at Clausthal University of Technology, Am Regenbogen 15
38678 Clausthal Zellerfeld, Germany, \email{pgallian@gmail.com}}
%
%
\maketitle

\abstract{We prove that the expressive power of first-order logic with team semantics plus contradictory negation does not rise beyond that of first-order logic (with respect to sentences), and that the totality atoms of arity $k+1$ are not definable in terms of the totality atoms of arity $k$. We furthermore prove that all first-order nullary and unary dependencies are strongly first order, in the sense that they do not increase the expressive power of first order logic if added to it.}
\section{Introduction}
\label{sec:intro}
In the last few years, team semantics \cite{hodges97, vaananen07} has proved itself to be a very powerful theoretical framework for the study of dependency notions and their interaction; and, furthermore, some intriguing potential applications of team semantics in the areas of  belief representation \cite{galliani12c, galliani13f}, social choice and physics \cite{abramsky13} and database theory \cite{kontinen13} have been noticed.

As a natural generalization of Tarski's semantics to the case of multiple assignments, team semantics allows to extend first-order logic in novel ways, in particular by adding to it \emph{dependency atoms} that specify complex patterns of dependence and independence between variables; and much of the research in the area so far has been dedicated to the comparison of the logics thus obtained. 

Many of these logics are much stronger than first-order logic itself -- for instance, dependence logic is as expressive as the existential fragment of second-order logic \cite{vaananen07}, and inclusion logic is as expressive as greatest fixed point logic \cite{gallhella13} -- but this needs not be the case. Indeed, as shown in \cite{galliani13e}, many nontrivial dependency notions, such as for instance the negations of functional dependence, inclusion, exclusion, and conditional independence, are \emph{strongly first-order} in the sense that they do not increase the expressive power of first-order logic if added to it. The \emph{totality atoms}, which assert that a certain tuple of variables takes all possible values in a team, are an especially interesting example of a strongly first-order dependency, and in this work we will study them in some depth.

It is important to emphasize here that these strongly first-order dependencies, despite not increasing the expressive power of first-order logic sentences, cannot be disposed of: even though every sentence containing them (but not other, stronger dependencies) is logically equivalent to some first-order sentence, the satisfaction conditions of \emph{formulas} containing them are not in general equivalent to the satisfaction conditions of any first-order formula with respect to team semantics. This disparity between the behaviour of formulas and that of sentences is one of the most intriguing phenomena of team semantics.

The study of team semantics (and, in particular, of strongly first-order dependencies) can thus be seen as an attempt to investigate the nature of the \emph{boundary} between first- and second-order logic; and, from a more practical point of view, dependencies which are strongly first-order are eminently treatable in that they do not increase the complexity of the logic. 

The purpose of this work is to further investigate the properties of strongly first-order dependencies and -- more in general -- of team semantics-based extensions of first-order logic whose expressive power is no greater than that of first-order logic proper. In Section \ref{sec:neg} we will investigate the effect of adding the contradictory negation operator to extensions of first-order logic by strongly first-order operator; then in Section \ref{sec:total} we will develop a \emph{hierarchy theorem} for totality atoms, and in Sections \ref{sec:zerary} and \ref{sec:unary} we will study dependency atoms of arity $0$ or $1$.
\section{Preliminaries}
\label{sec:prelim}
In this section we will briefly recall some fundamental definitions, as well as some results that we will need to use later in this work.

\begin{definition}[Team]
Let $\model$ be a first order model with domain $M$ and let $V$ be a set of variables. A \emph{team} $X$ over $\model$ with domain $\dom(X) = V$ is a set of assignments $s : V \rightarrow M$. 

Given such a team $X$ and a tuple $\tuple v$ of variables in $\dom(X)$, we write $X(\tuple v)$ for the relation $\{s(\tuple v) : s \in X\}$; and given a first-order formula $\theta$, we write $(X\upharpoonright \theta)$ for the team $\{s \in X : \model \models_s \theta\}$ obtained by taking only the assignments of $X$ which satisfy $\theta$ (according to Tarski's semantics).
\end{definition}

For the purposes of this work, we will only consider the so-called \emph{lax} version of team semantics, and we will only work with formula in negation normal form:
\begin{definition}
Let $\model$ be a first order model, let $X$ be a team over it, and let $\phi(\tuple v)$ be a first order formula in negation normal form and with free variables in $\tuple v \subseteq \dom(X)$. We say that $X$ \emph{satisfies} $\phi(\tuple v)$ in $\model$, and we write $M \models_X \phi(\tuple v)$, if and only if this can be deduced from the following rules: 
\begin{description}
\item[\textbf{TS-lit:}] For all first-order literals $\alpha$, $\model \models_X \alpha$ if and only if for all $s \in X$, $\model \models_s \alpha$ according to Tarski semantics;
\item[\textbf{TS-$\vee$:}] $\model \models_X \psi \vee \theta$ if and only if there exist $Y, Z \subseteq X$ such that $X = Y \cup Z$, $\model \models_Y \psi$ and $\model \models_Z \theta$; 
\item[\textbf{TS-$\wedge$:}] $\model \models_X \psi \wedge \theta$ if and only if $\model \models_X \psi$ and $\model \models_X \theta$; 
\item[\textbf{TS-$\exists$:}] $\model \models_X \exists v \psi$ if and only if there exists a function $F: X \rightarrow \parts(M) \backslash \{\emptyset\}$ such that, for $Y = X[F/v] = \{s[m/v] : m \in F(s)\}$, we have that $\model \models_Y \psi$; 
\item[\textbf{TS-$\forall$:}] $\model \models_X \forall v \psi$ if and only if $M \models_{X[M/v]} \psi$, where $X[M/v] = \{s[m/v] : s \in X, m \in M\}$.
\end{description}
A sentence $\phi$ is said to be \emph{true} in a model $\model$ if and only if $\model \models_{\startteam} \phi$; and in this case, we write $\model \models \phi$.
\end{definition}

The next result shows that, in the case of first-order logic, team semantics may indeed be reduced to Tarski's semantics: 

\begin{proposition}[\cite{vaananen07}]
\label{propo:flat}
For all first-order formulas $\phi$, all models $\model$ and all teams $X$, $\model \models_X \phi$ if and only if for all $s \in X$ we have that $\model \models_s \phi$ according to Tarski's semantics. In particular, for all first-order sentences $\phi$ we have that $\model \models_{\startteam} \phi$ if and only if $\model \models \phi$ according to Tarski's semantics.
\end{proposition}

However, team semantics allows us to extend first-order logic in novel ways, for instance by operators such as the \emph{intuitionistic implication} \cite{abramsky09}
\begin{description}
\item[\textbf{TS-intimp:}] $\model \models_X \phi \rightarrow \psi$ if and only if for all $Y \subseteq X$, $\model \models_Y \phi \Rightarrow \model \models_Y \psi$, 
\end{description}
the \emph{contradictory negation} \cite{vaananen07b}
\begin{description}
\item[\textbf{TS-$\sim$:}] $\model \models_X \sim \phi$ if and only if $\model \not \models_X \phi$, 
\end{description}
the \emph{classical disjunction} \cite{vaananen07}
\begin{description}
\item[\textbf{TS-$\sqcup$:}] $\model \models_X \phi \sqcup \psi$ if and only if $\model \models_X \phi$ or $\model \models_X \psi$,
\end{description}
or the \emph{possibility operator} \cite{galliani13e}
\begin{description}
\item[\textbf{TS-$\Diamond$:}] $\model \models_X \Diamond \phi$ iff there exists a $Y \subseteq X$, $Y \not = \emptyset$ s.t. $\model \models_Y \phi$
\end{description}
or by means of novel atoms corresponding to notions of constancy and functional dependence \cite{vaananen07}
\begin{description}
\item[\textbf{TS-con:}] $\model \models_X =\!\!(\tuple v)$ iff for all $s, s' \in X$, $s(\tuple v) = s'(\tuple v)$;
\item[\textbf{TS-fdep:}] $\model \models_X =\!\!(\tuple v, \tuple w)$ iff for all $s, s' \in X$, $s(\tuple v) = s'(\tuple v) \Rightarrow s(\tuple w) = s'(\tuple w)$, 
\end{description}
inclusion dependence \cite{galliani12}
\begin{description}
\item[\textbf{TS-inc:}] $\model \models_X \tuple v \subseteq \tuple w$ iff $X(\tuple v) \subseteq X(\tuple w)$
\end{description}
(conditional) independence \cite{gradel13}
\begin{description}
\item[\textbf{TS-ind:}] $\model \models_X \tuple v ~\bot_{\tuple u}~ \tuple w$ iff for all $s, s' \in X$ with $s(\tuple u) = s'(\tuple u)$ there exists a $s'' \in X$ with $s''(\tuple u \tuple v \tuple w) = s(\tuple u \tuple v)s'(\tuple w)$.
\end{description}
or totality \cite{abramsky13}:
\begin{description}
\item[\textbf{TS-all:}] $\model \models_X \all(\tuple v)$ iff $X(\tuple v) = M^{|\tuple v|}$.
\end{description}

More in general, all these atoms (and many more besides) can be seen as special cases of the following definition (\cite{kuusisto13}):

\begin{definition}[Dependency Notion]
Let $k \in \mathbb N$. A $k$-ary dependency notion $\Dep$ is a class, closed under isomorphisms, of models over the signature $\{R\}$, where $R$ is a $k$-ary relation symbol. For all models $\model$, all teams $X$, and all tuples $\tuple v$ of variables in the domain of $X$, 
\begin{equation*}
\model \models_X \Dep \tuple v \mbox{ if and only if } (M, X(\tuple v)) \in \Dep.
\end{equation*}
\end{definition}

Given a family $\DD$ of dependency notions, we will write $\FO(\DD)$ for the logic obtained by adding all $\mathbf D \in \DD$ to the language of first-order logic. We will indicate with $=\!\!(\cdot)$ the family of all constancy dependencies $=\!\!(\tuple v)$ of all arities, with $=\!\!(\cdot, \cdot)$ the family of all functional dependency atoms $=\!\!(\tuple v, \tuple w)$ of all arities, and with $\all$ the family of all totality atoms $\all (\tuple w)$ of all arities; and when necessary, we will indicate the arities as a subscript -- for instance, $=\!\!(\cdot)_1$ represents the unary constancy atoms $=\!\!(v)$ where $v$ is a single variable, and $=\!\!(\cdot, \cdot)_{2,2}$ represents the functional dependency atoms of the form $=\!\!(v_1 v_2, w_1 w_2)$. 

The following notion of \emph{definability} is of central importance for the study of team semantics: 
\begin{definition}[Definability]
Let $\Dep$ be a $k$-ary dependency notion and let $\DD$ be a class of dependency notions. Then we say that $\Dep$ is \emph{definable} through $\DD$ if there exists a formula $\theta(\tuple v) \in \FO(\DD)$ over the empty vocabulary, where $\tuple v = v_1 \ldots v_k$ is a tuple of $k$ distinct variables, such that 
\begin{equation*}
\model \models_X \Dep \tuple v \mbox{ if and only if } \model \models_X \theta(\tuple v)
\end{equation*}
for all models $\model$ and teams $X$ whose domain contains $\tuple v$.
\end{definition}

It is easy to see that $\FO(=\!\!(\cdot)) = \FO(=\!\!(\cdot)_1)$: indeed, for any $k$-tuple $\tuple v = v_1 \ldots v_k$ of variable it is trivial to check that $=\!\!(\tuple v) \equiv \bigwedge_{i=1}^k =\!\!(v_i)$, and hence $=\!\!(\cdot)_k$ is definable through $=\!\!(\cdot)_1$. On the other hand, in \cite{durand11} it was shown that
\begin{theorem}
For all $k \in \mathbb N$, $\FO(=\!\!(\cdot, \cdot)_{k,1}) \subsetneq \FO(=\!\!(\cdot, \cdot)_{k+1, 1})$,\footnote{To be more precise, this results holds if we are allowing models over all signatures. The case in which only models over the empty signature are considered is yet open.}
\end{theorem}
in \cite{galliani13b} it was shown that a similar result holds for independence atoms, and in \cite{hannula14} it was shown that the same may be said in the case of inclusion atoms too.

What about totality dependencies? We will address this question in Section \ref{sec:total}.

All dependencies that we mentioned so far are \emph{first-order} in the following sense:
\begin{definition}[First-Order Dependency Notion]
A $k$-ary dependency notion $\Dep$ is first-order if and only if there exists a first order formula $\Dep^*$ on the signature $\{R\}$ (for $R$ $k$-ary) such that 
\begin{equation*}
\Dep = \{(M, R) : (M, R) \models \Dep^*\}.
\end{equation*}
\end{definition}
It is easy to see that if $\Dep$ is first-order then $\model \models_X \Dep \tuple v \Leftrightarrow (M, X(\tuple v)) \models \Dep^*$; but owing to the higher-order nature of team semantics (and in particular, to the second-order quantification implicit in its rules for disjunctions and existential quantifiers) it does not follow from this that these dependencies do not increase the expressive power of first-order logic. For instance, the $\FO(=\!\!(\cdot, \cdot)_{1,1})$-sentence
\begin{equation*}
\exists x \forall y \exists z (=\!\!(z, y) \wedge z \not = x)
\end{equation*}
is true in a model $\model$ if and only if it is infinite, even though $=\!\!(\cdot, \cdot)_{1,1}$ is first-order and corresponds to the sentence $\forall x y y' (R xy \wedge R xy' \rightarrow y = y')$.

Therefore, the question arises of whether there exist interesting dependency notions for which this is not the case. More formally, one may ask if there exist nontrivial dependencies which are \emph{strongly first-order} in the following sense:
\begin{definition}[Strongly First Order Dependencies]
A $k$-ary dependency $\Dep$ is \emph{strongly first order} if every sentence of $\FO(\Dep)$ is equivalent to some sentence of $\FO$. Similarly, a family of dependencies $\DD$ is strongly first order if every sentence of $\FO(\DD)$ is equivalent to some sentence of $\FO$.
\end{definition}

In \cite{galliani13e}, a positive answer was found for the above question.

\begin{definition}
A dependency notion $\Dep$ is \emph{upwards-closed} if $(M, R) \in \Dep, R \subseteq S \Rightarrow (M, S) \in \Dep$.
\end{definition}
\begin{theorem}[\cite{galliani13e}]
\label{thm:uwc}
Let $\DD$ be a family of upwards-closed first-order dependencies. Then $\{=\!\!(\cdot)\} \cup \DD$ is strongly first order.
\end{theorem}
As a consequence, it was shown that -- for instance -- all the following dependencies are strongly first-order for all arities of $\tuple v$ and $\tuple w$, as is any set containing them (and the constancy atoms $=\!\!(\cdot)$): 
\begin{description}
\item[\textbf{TS-nonempty:}] $\model \models_X \nonempty$ iff $X \not = \emptyset$; 
\item[\textbf{TS-ncon:}] $\model \models_X \not=\!\!(\tuple v)$ iff there exist $s, s' \in X$ such that $s(\tuple v) \not = s'(\tuple v)$; 
\item[\textbf{TS-ndep:}] $\model \models_X \not =\!\!(\tuple v, \tuple w)$ iff there exist $s, s' \in X$ with $s(\tuple v) = s'(\tuple v)$ but $s(\tuple w) \not = s'(\tuple w)$; 
\item[\textbf{TS-geq:}] For all $n \in \mathbb N$, $\model \models_X |\tuple v| \geq n$ iff $|X(\tuple v)| \geq n$; 
\item[\textbf{TS-all:}] $\model \models_X \all(\tuple v)$ iff $X(\tuple v) = M^{|\tuple v|}$; 
\item[\textbf{TS-$\not \subseteq$:}] $\model \models_X \tuple v \not \!\subseteq \tuple w$ iff there exists some $s \in X$ such that for all $s' \in X$, $s(\tuple v) \not = s'(\tuple w)$; 
\item[\textbf{TS-$\not \!\!\bot$:}] $\model \models_X \tuple v \not \!\!\bot_{\tuple u} \tuple w$ iff there exist $s, s' \in X$ with $s(\tuple u) = s'(\tuple u)$ but such that for all $s'' \in X$, $s''(\tuple u \tuple v \tuple w) \not = s(\tuple u \tuple v) s'(\tuple w)$.
\end{description}
The last two dependencies are not upwards-closed, but as shown in \cite{galliani13e} they are definable in terms of constancy atoms and first-order, upwards-closed dependencies. 

We conclude this section by mentioning a few shorthands and results that we will need to use in the rest of this work: 

\begin{definition}[$\top$, $\bot$]
Let $v$ be any variable. Then we write $\top$ for $\forall v (v = v)$ and $\bot$ for $\exists v (v \not = v)$.
\end{definition}
\begin{proposition}
For all models $\model$ and teams $X$, $\model \models_X \top$; and furthermore, $\model \models_X \bot$ if and only if $X = \emptyset$.
\end{proposition}
\begin{proof}
Obvious.
\end{proof}

%
\begin{definition}[Dual Negation]
Let $\phi$ be a first-order formula in negation normal form. Then we write $\lnot \phi$ as a shorthand for the formula thus obtained: 
\begin{itemize}
\item If $\phi$ is a positive literal $R \tuple t$ or $t_1 = t_2$, $\lnot \phi$ is its negation (that is, $\lnot R \tuple t$ or $t_1 \not = t_2)$; 
\item If $\phi$ is a negative literal $\lnot R \tuple t$ or $t_1 \not = t_2$, $\lnot \phi$ is the corresponding positive literal (that is, $R \tuple t$ or $t_1 = t_2)$; 
\item $\lnot (\phi \vee \psi) = (\lnot \phi) \wedge (\lnot \psi)$; 
\item $\lnot (\phi \wedge \psi) = (\lnot \phi) \vee (\lnot \psi)$; 
\item $\lnot (\exists v \phi) = \forall v (\lnot \psi)$; 
\item $\lnot (\forall v \phi) = \exists v (\lnot \psi)$; 
\end{itemize}
\end{definition}
It is not difficult to see, by structural induction on $\phi$, that 
\begin{proposition}
For all first-order formulas $\phi$, all models $\model$ and all teams $X$, $\model \models_X \lnot \phi$ if and only if for all $s \in X$ we have that $\model \models_s \lnot \phi$ according to Tarski's semantics. 
\end{proposition}
\begin{definition}[$\phi \upharpoonright \theta$]
Let $\DD$ be any class of dependencies, let $\phi \in \FO(\DD)$ and let $\theta \in \FO$. Then we write $\phi \upharpoonright \theta$ as a shorthand for 
\begin{equation*}
(\lnot \theta) \vee (\theta \wedge \phi)
\end{equation*}
\end{definition}
\begin{proposition}[\cite{galliani13e}]
Let $\DD$ be any class of dependencies, let $\phi \in \FO(\DD)$ and let $\theta \in \FO$. Then for all suitable models $\model$ and teams $X$, 
\begin{equation*}
\model \models_X \phi \upharpoonright \theta \mbox{ if and only if } \model \models_{X \upharpoonright \theta} \phi.
\end{equation*}
\end{proposition}
\begin{definition}[Flattening]
Let $\DD$ be any class of dependencies and let $\phi \in \FO(\nonempty)$. Then we define its \emph{flattening} $\phi^f$ as the first-order formula obtained by substituting all atoms $\Dep \tuple v$ in it with $\top$.
\end{definition}
\begin{lemma}
\label{lemma:flatup}
For all classes of dependencies $\DD$, models $\model$, teams $X$, and formulas $\phi \in \FO(\DD)$, if $\model \models_X \phi$ then $\model \models_X \phi^f$.
\end{lemma}
\begin{proof}
Trivial.
\end{proof}
\begin{lemma}[\cite{galliani13e}]
\label{lemma:raise}
Let $\DD$ be a class of \emph{upwards-closed} (but not necessarily first-order) dependencies. Then for all models $\model$, teams $X$ and $Y$ such that $X \subseteq Y$, and $\phi \in \FO(\DD)$, if $\model \models_X \phi$ and $\model \models_Y \phi^f$ then $\model \models_Y \phi$.
\end{lemma}
\begin{lemma}[\cite{galliani13e}]
\label{lemma:constout} 
Let $\DD$ be a class of dependencies (not necessarily first-order or upwards-closed) and let $\phi(\tuple v)$ be a $\FO(=\!\!(\cdot), \DD)$ formula. Then $\phi(\tuple v)$ is equivalent to some formula of the form $\exists \tuple w(=\!\!(\tuple w) \wedge \psi(\tuple w, \tuple v))$, where $\psi \in \FO(\DD)$ contains the exactly the same instances of $\Dep$-atoms (for all $\Dep \in \DD$) that $\phi$ does, and in the same number.
\end{lemma}

The following simple result - which allows us to add, essentially for free, the classical disjunction $\sqcup$ to our language -- will also be of some use in the rest of this work:
\begin{proposition}
\label{propo:dnf}
Let $\DD$ be any class of dependencies and let $\phi \in \FO(\DD, \sqcup)$. Then $\phi$ is equivalent to some formula of the form $\bigsqcup_{i=1}^n \psi_i$, where all $\psi_i$ are in $\FO(\DD)$.
\end{proposition}
\begin{proof}
It suffices to show that the $\sqcup$ connective commutes with all other connectives: 
\begin{itemize}
\item $(\psi \sqcup \theta) \vee \chi \equiv (\psi \vee \chi) \sqcup (\theta \vee \chi)$:  Suppose that $\model \models_X (\psi \sqcup \theta) \vee \chi$. Then $X = Y \cup Z$ for two $Y$, $Z$ such that $\model \models_Y \psi \sqcup \theta$ and $\model \models_Z \chi$. By the satisfaction conditions for $\sqcup$, we have that $\model \models_Y \psi$ or $\model \models_Y \theta$. In the first case we have that $\model \models_X \psi \vee \chi$ and in the second case we have that $\model \models_X \theta \vee \chi$, so in either case $\model \models_X (\psi \vee \chi) \sqcup (\theta \vee \chi)$.

Conversely, suppose that $\model \models_X (\psi \vee \chi) \sqcup (\theta \vee \chi)$. Then $\model \models_X (\psi \vee \chi)$ or $\model \models_X (\theta \vee \chi)$. In the first case, we have that $X = Y \cup Z$ for two $Y$ and $Z$ such that $\model \models_Y \psi$ and $\model \models_Z \chi$; but then $\model \models_Y \psi \sqcup \theta$ too, and thus $\model \models_X (\psi \sqcup \theta) \vee \chi$. The case in which $\model \models_X (\theta \vee \chi)$ is dealt with analogously.
\item $(\psi \sqcup \theta) \wedge \chi \equiv (\psi \wedge \chi) \sqcup (\theta \wedge \chi)$: $\model \models_X (\psi \sqcup \theta) \wedge \chi$ iff ($\model \models_X \psi$ or $\model \models_X \theta$) and $\model \models_X \chi$ iff ($\model \models_X \psi$ and $\model \models_X \chi$) or ($\model \models_X \theta$ and $\model \models_X \chi$) iff $\model \models_X (\psi \wedge \chi) \sqcup (\theta \wedge \chi)$.
\item $\exists v (\psi \sqcup \theta) \equiv (\exists v \psi) \sqcup (\exists v \theta)$: Suppose that $\model \models_X \exists v (\psi \sqcup \theta)$. Then there exists a choice function $F$ such that $\model \models_{X[F/v]} \psi$ or $\model \models_{X[F/v]} \theta$. In the first case we have that $\model \models_X \exists v \psi$, and in the second case we have that $\model \models_X \exists v \theta$; so in either case $\model \models_X (\exists v \psi) \sqcup (\exists v \theta)$. 

Conversely, suppose that $\model \models_X  (\exists v \psi) \sqcup (\exists v \theta)$. If $\model \models_X (\exists v \psi)$ then there is a $F$ such that $\model \models_{X[F/v]} \psi$, and therefore $\model \models_{X[F/v]} \psi \sqcup \theta$, and therefore $\model \models_X \exists v (\psi \sqcup \theta)$; and similarly, if  $\model \models_X (\exists v \theta)$ there is a $F$ such that $\model \models_{X[F/v]} \theta$, and therefore $\model \models_{X[F/v]} \psi \sqcup \theta$, and therefore $\model \models_X \exists v (\psi \sqcup \theta)$.
\item $\forall v (\psi \sqcup \theta) \equiv (\forall v \psi) \sqcup (\forall v \theta)$: $\model \models_X \forall v (\psi \sqcup \theta)$ iff $\model \models_{X[M/v]} (\psi \sqcup \theta)$ iff ($\model \models_{X[M/v]} \psi$ or $\model \models_{X[M/v]} \theta$) iff ($\model \models_X \forall v \psi$ or $\model \models_X \forall v \theta$) iff $\model \models_X (\forall v \psi) \sqcup (\forall v \theta)$. 
\end{itemize}
\end{proof}
\begin{lemma}
For all models $\model$ and sentences $\phi_1, \phi_2 \in \FO$, 
\begin{equation*}
\model \models \phi_1 \sqcup \phi_2 \Leftrightarrow \model \models \phi_1 \vee \phi_2.
\end{equation*}
\end{lemma}
\begin{proof}
Suppose that $\model \models \phi_1 \sqcup \phi_2$. Then, by definition, $\model \models_{\startteam} \phi_i$ for some $i \in \{1,2\}$. Suppose, without loss of generality, that $\model \models_{\startteam} \phi_1$; then since $\phi_2$ is first-order we have that $\model \models_\emptyset \phi_2$, and hence $\model \models_{\startteam} \phi_1 \vee \phi_2$. The case for $\model \models_{\startteam} \phi_2$ is analogous. Conversely, suppose that $\model \models_{\startteam } \phi_2 \vee \phi_2$: then $\startteam = Y \cup Z$ for two $Y, Z \subseteq \startteam$ such that $Y \cup Z = \startteam$, $\model \models_Y \phi_1$ and $\model \models_Z \phi_2$. Then $Y = \startteam$ or $Z = \startteam$, and hence $\model \models \phi_1$ or $\model \models \phi_2$ and finally $\model \models \phi_1 \sqcup \phi_2$, as required.
\end{proof}
\begin{corollary}
\label{coro:addsqcup}
Let $\mathcal D$ be a strongly first-order class of dependencies. Then every sentence of $\FO(\mathcal D, \sqcup)$ is equivalent to some sentence of $\FO$.
\end{corollary}
\begin{proof}
Let $\phi \in \FO(\mathcal D, \sqcup)$. As per the above results, we may assume that $\phi$ is of the form $\bigsqcup_i \psi_i$, where all $\psi_i$ are $\FO(\mathcal D)$-sentences, and hence equivalent to first-order sentences $\psi'_i$. Now let $\phi' = \bigvee_i \psi'_i$. 
\end{proof}
\section{On the Contradictory Negation}
\label{sec:neg}
It is known from \cite{vaananen07b} that \emph{team logic} $\FO(=\!\!(\cdot, \cdot), \sim)$, that is, the logic obtained by adding the contradictory negation \emph{and} functional dependency conditions (of all arities) to the language of first-order logic, is as expressive as second-order logic over sentences; and, furthermore, in \cite{kontinennu09} it was shown that all second-order properties of teams correspond to the satisfaction conditions of team logic sentences. 

But what if we add the contradictory negation to weaker extensions of first-order logic? Or, for that matter, what if we consider $\FO(\sim)$, that is, the logic obtained by adding \emph{only} the contradictory negation to the language of first-order logic? 

In this section, we will prove that 
\begin{enumerate}
\item Both $\FO(\sim, =\!\!(\cdot))$ and $\FO(\sim, \not =\!\!(\cdot))$ are equivalent to full team logic;
\item $\FO(\sim) = \FO(\nonempty, \sqcup) = \FO(\nonempty, \sqcup, \sim)$;
\item Every sentence of $\FO(\nonempty, \sqcup)$ is equivalent to some first-order sentence.
\end{enumerate}
Thus, the contradictory negation alone does not suffice to bring the expressive power of our logic beyond that of first-order logic, but as soon as we add even simple strongly first-order dependencies such as constancy or non-constancy we obtain the full expressive power of second-order logic. 

\begin{lemma}
$\FO(\sim, =\!\!(\cdot)) = \FO(\sim, \not = \!\!(\cdot))$.
\end{lemma}
\begin{proof}
It suffices to observe that, for any tuple $\tuple v$ of variables, $\not = \!\!(\tuple v)$ is logically equivalent to $\sim =\!\!(\tuple v)$ and $=\!\!(\tuple v)$ is logically equivalent to $\sim \not = \!\!(\tuple v)$.
\end{proof}
\begin{lemma}
For any two tuple $\tuple v$, $\tuple w$ of variables, the functional dependence atom $=\!\!(\tuple v, \tuple w)$ is definable in $\FO(\sim, =\!\!(\cdot))$.
\end{lemma}
\begin{proof}
Consider the formula 
\begin{equation}
\label{eq:eqdep}
\sim (\exists \tuple p \tuple q_1 \tuple q_2 (=\!\!(\tuple p) \wedge =\!\!(\tuple q_1) \wedge = \!\!(\tuple q_2) \wedge \tuple q_1 \not = \tuple q_2 \wedge \sim(\tuple v \tuple w \not = \tuple p \tuple q_1) \wedge \sim (\tuple v \tuple w \not = \tuple p \tuple q_2)).
\end{equation}
It is easy to check that Equation (\ref{eq:eqdep}) is logically equivalent to $=\!\!(\tuple v, \tuple w)$, as required.
\end{proof}
\begin{corollary}
$\FO(\sim, =\!\!(\cdot)) = \FO(\sim, \not = \!\!(\cdot)) = \FO(\sim, =\!\!(\cdot, \cdot)) =$ Team Logic.
\end{corollary}

So far so good. But what can we say about $\FO(\sim)$? In what follows, we will prove that this logic is \emph{not} more expressive than first-order logic over sentences; indeed, it is equivalent to $\FO(\nonempty, \sqcup)$.

\begin{lemma}
Let $\psi \in \FO(\nonempty)$ and let $\theta \in \FO$. Then $\sim(\psi \upharpoonright \theta)$ is logically equivalent to $((\sim \psi) \upharpoonright \theta)$. 
\end{lemma}
\begin{proof}
Suppose that $\model \models_X \sim(\psi \upharpoonright \theta)$. Then for $Y  = \{s \in X : \model \models_s \theta\}$ we have that $\model \not \models_Y \psi$. But then $\model \models_Y \sim \psi$, and thus $\model \models_X ((\sim \psi) \upharpoonright \theta)$. 

Conversely, suppose that $\model \models_X ((\sim \psi) \upharpoonright \theta)$. Then for $Y$ as above we have that $\model \not \models_Y \psi$; and therefore, $\model \not \models_X (\psi \upharpoonright \theta)$, and in conclusion $\model \models_X \sim (\psi \upharpoonright \theta)$. 
\end{proof}
\begin{lemma}
Let $\phi \in \FO(\nonempty)$. Then $\sim \phi$ is equivalent to some formula in $\FO(\nonempty, \sqcup)$. 
\end{lemma}
\begin{proof}
We proceed by structural induction on $\phi$. 
\begin{enumerate}
\item Suppose that $\phi$ is a first-order formula (not necessarily a literal). Then $\sim \phi$ is logically equivalent to $\nonempty \upharpoonright (\lnot \phi)$. Indeed, suppose that $\model \models_X \sim \phi$: then, since $\phi$ is first-order, there exists a $s \in X$ such that $\model \not \models_s \phi$ according to Tarski's semantics. But then $s \in X \upharpoonright (\lnot \phi)$, and thus $\model \models_X \nonempty \upharpoonright (\lnot \phi)$. 

Conversely, suppose that $\model \models_X \nonempty \upharpoonright (\lnot \phi)$. Then the set $X \upharpoonright (\lnot \phi)$ is not empty, and therefore there exists some $s \in X$ which satisfies $\lnot \phi$ according to Tarski's semantics, and finally $\model \not \models_X \phi$.
\item $\sim \nonempty$ is easily seen to be equivalent to $\bot$, which is true only in the empty team.
\item Suppose that $\phi$ is of the form $(\psi \vee \theta)$. Then $\sim \phi$ is logically equivalent to 
\begin{equation}
\label{eq:eq1}
((\sim \psi) \upharpoonright \psi^f) \sqcup ( (\sim \theta) \upharpoonright \theta^f) \sqcup \sim (\psi^f \vee \theta^f).
\end{equation}
Indeed, suppose that $\model \models_X \sim(\psi \vee \theta)$. Then it is not the case that $X = Y \cup Z$ for two $Y$, $Z$ such that $\model \models_Y \psi$ and $\model \models_Z \theta$. In particular, take $Y = X \upharpoonright \psi^f$ and $Z = X \upharpoonright \theta^f$: then $Y \cup Z \not = X$, and hence $\model \models_X \sim (\psi^f \vee \theta^f)$, or $\model \not \models_Y \psi$, and hence $\model \models_X ((\sim \psi) \upharpoonright \psi^f)$, or $\model \not \models_Z \theta$, and hence $\model \models_X ((\sim \theta) \upharpoonright \theta^f)$. 

Conversely, suppose that $\model \models_X (\psi \vee \theta)$. Then $X = Y \cup Z$ for two $Y, Z$ such that $\model \models_Y \psi$ and $\model \models_Y \theta$. Now take $Y' = X \upharpoonright \psi^f$ and $Z' = X \upharpoonright \theta^f$: by Proposition \ref{propo:flat} we have that $\model \models_Y \psi^f$ and $\model \models_Z \theta^f$, by Lemma \ref{lemma:flatup} we have that $Y \subseteq Y'$ and $Z \subseteq Z'$, and thus $X = Y' \cup Z'$, and by Lemma \ref{lemma:raise} we have that $\model \models_{Y'} \psi$ and $\model \models_{Z'} \theta$. Therefore $\model \not \models_X (\sim \psi) \upharpoonright \psi^f$, $\model \not \models_X (\sim \theta) \upharpoonright \theta^f$, and $\model \models_X (\psi^f \vee \theta^f)$, so in conclusion $X$ does not satisfy Equation (\ref{eq:eq1}).
\item Suppose that $\phi$ is of the form $(\psi \wedge \theta)$. Then $\sim \phi$ is logically equivalent to $(\sim \psi) \sqcup (\sim \theta)$. 
\item Suppose that $\phi$ is of the form $(\exists v \psi)$. Then $\sim \phi$ is logically equivalent to 
\begin{equation}
\label{eq:eq2}
\sim (\exists v \psi^f) \sqcup \forall v ( (\sim \psi) \upharpoonright \psi^f)
\end{equation}
Indeed, suppose that $\model \models_X \sim(\exists v \psi)$ and $\model \models_X \exists v \psi^f$, and consider the choice function $F$ such that $F(s) = \{m : \model \models_{s[m/v]} \psi^f\}$. $F(s)$ is nonempty for all $s \in X$, since $\model \models_X \exists v \psi^f$; and therefore, by hypothesis, $\model \not \models_{X[F/v]} \psi$. But by construction, we have that $X[F/v] = X[M/v]\upharpoonright \psi^f$, and thus $\model \not \models_{X[M/v]} \psi \upharpoonright \psi^f$, and finally  $\model \models_X \forall v ((\sim \psi) \upharpoonright \psi^f)$.

Conversely, suppose that there exists a choice function $F: X \rightarrow \parts(M) \backslash \{\emptyset\}$ such that $\model \models_{X[F/v]} \psi$. Then in particular $\model \models_{X[F/v]} \psi^f$, and hence $\model \models_X \exists v \psi^f$ and $\model \not \models_X \sim (\exists v \psi^f)$; and furthermore, we have that $X[F/v] \subseteq X[M/v] \upharpoonright \psi^f$, and therefore $\model \models_{X[M/v]} \psi \upharpoonright \psi^f$ and $\model \not \models_{X[M/v]} (\sim \psi) \upharpoonright \psi^f$. So in conclusion the team $X$ does not satisfy Equation (\ref{eq:eq2}).
\item Suppose that $\phi$ is of the form $(\forall v \psi)$. Then $\sim \phi$ is logically equivalent to $\forall v \sim \psi$: indeed, $\model \models_X \sim \phi$ iff $\model \not \models_X \forall v \psi$ iff $\model \not \models_{X[M/v]} \psi$ iff $\model \models_{X[M/v]} \sim \psi$ iff $\model \models_X \forall v \sim \psi$. 
%
\end{enumerate}
\end{proof}
We are now equipped to prove the main result of this section: 
\begin{theorem}
Let $\phi \in \FO(\nonempty, \sqcup)$. Then $\sim \phi$ is equivalent to some formula in $\FO(\nonempty, \sqcup)$. 
\end{theorem}
\begin{proof}
By Proposition \ref{propo:dnf}, we may assume that $\phi$ is of the form $\bigsqcup_{i=1}^n \psi_i$, where each $\psi_i$ is in $\FO(\nonempty)$. Thus, $\sim \phi$ is logically equivalent to $\bigwedge_{i=1}^n (\sim \psi_i)$; and by the above lemma, if $\psi_i$ is in $\FO(\nonempty)$ then $\sim \psi_i$ is in $\FO(\nonempty, \sqcup)$, as required.
\end{proof}
The two following corollaries then follow at once:
\begin{corollary}
$\FO(\nonempty, \sqcup, \sim) = \FO(\nonempty, \sqcup)$.
\end{corollary}
\begin{corollary}
$\FO(\sim) \subseteq \FO(\nonempty, \sqcup)$.
\end{corollary}

We still need to show the other direction of the equivalence between $\FO(\nonempty, \sqcup)$ and $\FO(\sim)$:
\begin{proposition}
$\FO(\nonempty, \sqcup) \subseteq \FO(\sim)$.
\end{proposition}
\begin{proof}
It suffices to show that the nonemptiness atom and the classical disjunction are definable in $\FO(\sim)$. As for the former, observe that $\model \models_X \sim \bot$ if an only if $X$ is nonempty; and for the latter, observe that $\phi \sqcup \psi$ is logically equivalent to $\sim ((\sim \phi) \wedge (\sim \psi))$.
\end{proof}

Putting everything together, we have that 
\begin{theorem}
$\FO(\sim) = \FO(\nonempty, \sqcup)$.
\end{theorem}

Finally, we need to prove that every sentence of $\FO(\nonempty, \sqcup)$ is equivalent to some first-order sentence. But this is immediate: 
\begin{theorem}
Let $\phi \in \FO(\nonempty, \sqcup)$ be a sentence. Then $\phi$ is logically equivalent to some $\phi' \in \FO$. 
\end{theorem}
\begin{proof}
By Proposition \ref{propo:dnf} we may assume that $\phi$ is of the form $\bigsqcup_{i=1}^n \psi_i$, where each $\psi_i$ is a sentence in $\FO(\nonempty)$. But then by Theorem \ref{thm:uwc}, each $\psi_i$ is equivalent to some first-order sentence $\psi'_i$, and thus $\phi$ is equivalent to the first-order sentence $\bigvee_{i=1}^n \psi'_i$.
\end{proof}
\begin{corollary}
The constancy and inconstancy atoms are \emph{not} definable in $\FO(\nonempty, \sqcup)$.
\end{corollary}
\begin{proof}
If they were then we would have that $\FO(=\!\!(\cdot, \cdot), \sim) \subseteq \FO(\nonempty, \sqcup)$; but this is not possible, because $\FO(\nonempty, \sqcup)$ is strongly first-order and $\FO(=\!\!(\cdot, \cdot), \sim)$ is as strong as second-order logic.
\end{proof}
\section{Arity Hierarchies for Totality Atoms}
\label{sec:total}
In this section we will investigate the properties of the $k$-ary \emph{totality atoms} $\all_k$, and we we establish a strict \emph{arity hierarchy} for them.

Let us begin by generalizing a notion from \cite{galliani13e}: 
\begin{definition}[$\gamma$-boundedness]
\label{defin:bounded}
Let $\gamma: \mathbb N \rightarrow \mathbb N$ be a function. Then a dependency notion $\Dep$ is said to be $\gamma$-bounded if for all finite models $\model$ and teams $X$, if $\model \models_X \Dep$ then there exists a subteam $Y \subseteq X$, $|Y| \leq \gamma(|M|)$, such that $\model \models_Y \Dep$.
\end{definition}
\begin{proposition}
All $k$-ary dependencies $\Dep$ are $|M|^k$-bounded.
\end{proposition}
\begin{proof}
Suppose that $\model \models_X \Dep \tuple v$. Then $(M, X(\tuple v)) \in \Dep$; and since $X(\tuple v) \subseteq M^k$, it is clear that $|X(\tuple v)| \leq |M|^k$. Now for any $\tuple m \in X(\tuple v)$, let $s_{\tuple m} \in X$ be such that $s_{\tuple m}(\tuple v) = \tuple m$, and let $Y = \{s_{\tuple m} : \tuple m \in X(\tuple v)\}$. Then $|Y| \leq |M|^k$ and $Y(\tuple v) = X(\tuple v)$, and thus $\model \models_Y \Dep \tuple v$.
\end{proof}
\begin{theorem}
Let $\DD = \{\Dep_i : i \in I\}$ be a class of upwards-closed dependencies, for every $\Dep_i \in \DD$ let $\gamma_{i} : \mathbb N \rightarrow \mathbb N$ be such that $\Dep_i$ is $\gamma_{i}$-bounded, let $\phi \in \FO(\DD)$ be such that every $\Dep_i$ occurs $k_i$ times, and let $\nu_{\phi}(n) = \Sigma_{i \in I} k_i \gamma_i(n)$. Then $\phi$ is $\nu_{\phi}$-bounded, in the sense that 
\begin{equation*}
\model \models_X \phi \Rightarrow \exists Y \subseteq X, |Y| \leq \nu_{\phi}(|M|), \model \models_Y \phi
\end{equation*}
for all finite models $\model$ and all teams $X$.
\end{theorem}
\begin{proof}
The proof is by induction, and mirrors the analogous proof from \cite{galliani13e}.
\begin{enumerate}
\item If $\phi$ is a first order literal then it is $0$-bounded (since the empty team satisfies it), as required.
\item If $\phi$ is an atom $\mathbf D \tuple x$ then the statement follows at once from the definitions of boundedness.
\item Let $\phi$ be a disjunction $\psi_1 \vee \psi_2$ then $\nu_\phi = \nu_{\psi_1} + \nu_{\psi_2}$. Suppose now that $\model \models_X \psi_1 \vee \psi_2$: then $X = X_1 \cup X_2$ for two $X_1$ and $X_2$ such that $\model \models_{X_1} \psi_1$ and $\model \models_{X_2} \psi_2$. This implies that there exist $Y_1 \subseteq X_1$, $Y_2 \subseteq X_2$ such that $\model \models_{Y_1} \psi_1$ and $\model \models_{Y_2} \psi_2$, $|Y_1| \leq \nu_{\psi_1}(|M|)$ and $|Y_2| \leq \nu_{\psi_2}(|M|)$. But then $Y = Y_1 \cup Y_2$ satisfies $\psi_1 \vee \psi_2$ and has at most $\nu_{\psi_1}(|M|) + \nu_{\psi_2}(|M|)$ elements. 
\item If $\phi$ is a conjunction $\psi_1 \wedge \psi_2$ then, again, $\nu_{\phi} = \nu_{\psi_1} + \nu_{\psi_2}$. Suppose that $\model \models_X \psi_1 \wedge \psi_2$: then $\model \models_X \psi_1$ and $\model \models_X \psi_2$, and therefore by Lemma \ref{lemma:flatup} $\model \models_X \psi_1^f$ and $\model \models_X \psi_2^f$; and, by induction hypothesis, there exist $Y_1, Y_2 \subseteq X$ with $|Y_1| \leq \nu_{\psi_1}(|M|)$, $|Y_2| \leq \nu_{\psi_2}(|M|)$, $\model \models_{Y_1} \psi_1$ and $\model \models_{Y_2} \psi_2$. Now let $Y = Y_1 \cup Y_2$: since $Y \subseteq X$, by Proposition \ref{propo:flat} $\model \models_Y \psi_1^f$ and $\model \models_Y \psi_2^f$. But $Y_1, Y_2 \subseteq Y$, and therefore by Lemma \ref{lemma:raise} $\model \models_Y \psi_1$ and $\model \models_Y \psi_2$, and in conclusion $\model \models_Y \psi_1 \wedge \psi_2$.
\item If $\phi$ is of the form $\exists v \psi$ then $\nu_{\phi} = \nu_{\psi}$. Suppose that $\model \models_X \exists v \psi$: then for some $F$ we have that $\model \models_{X[F/v]} \psi$, and therefore by induction hypothesis there exists a $Z \subseteq X[F/v]$ with $|Z| \leq \nu_{\psi}(|M|)$ such that $\model \models_{Z} \psi$. For any $h \in Z$, let $\mathfrak f (h)$ be a $s \in X$ such that $h \in s[F/v] = \{s[m/v] : m \in F(s)\}$,\footnote{Since $Z \subseteq X[F/v]$, such a $s$ always exists. Of course, there may be multiple ones; in that case, we just pick arbitrarily one.} and let $Y = \{\mathfrak f(h) : h \in Z\}$. Now $Z \subseteq Y[F/v] \subseteq X[F/v]$. Since $\model \models_{X[F/v]} \psi^f$ and $Y[F/v] \subseteq X[F/v]$, we have that $\model \models_{Y[F/v]} \psi^f$; and since $\model \models_{Z} \psi$, this implies that $\model \models_{Y[F/v]} \psi$ and that $\model \models_Y \exists v \psi$. Furthermore $|Y| \leq |Z| \leq \nu_{\phi}(|M|)$, as required.
\item If $\phi$ is of the form $\forall v \psi$ then, again, $\nu_\phi = \nu_\psi$. Suppose that $\model \models_{X[M/v]} \psi$: again, by induction hypothesis there is a $Z \subseteq X[M/v]$ with $|Z| \leq \nu_\psi(|M|)$ and such that $\model \models_{Z} \psi$. For any $h \in Y$, let $\mathfrak g (h)$ pick some $s \in X$ which agrees with $h$ on all variables except possibly $v$, and let $Y = \{\mathfrak g(h) : h \in Z\}$. Similarly to the previous case, $Z \subseteq Y[M/v] \subseteq X[M/v]$: therefore, since $\model \models_{X[M/v]} \psi^f$ we have that $\model \models_{Y[M/v]} \psi^f$, and since $\model \models_{Z} \psi$ we have that $\model \models_{Y[M/v]} \psi$. So in conclusion $\model \models_Y \forall v \psi$, as required, and $|Y| \leq |Z| \leq \nu_{\phi}(M)$. 
\end{enumerate}
\end{proof}

Using some care, we can extend this result to the case of $\FO(=\!\!(\cdot), \DD, \sqcup)$: 
\begin{theorem}
Let $\DD = \{\Dep_i : i \in I\}$ be a class of upwards-closed dependencies, for every $\Dep_i \in \DD$ let $\gamma_{i} : \mathbb N \rightarrow \mathbb N$ be such that $\Dep_i$ is $\gamma_{i}$-bounded, let $\phi \in \FO(=\!\!(\cdot), \DD, \sqcup)$ be such that every $\Dep_i$ occurs $k_i$ times, and let $\nu_\phi(n) = \Sigma_{i \in I} k_i \gamma_i(n)$. Then $\phi$ is $\nu_\phi$-bounded, in the sense that 
\begin{equation*}
\model \models_X \phi \Rightarrow \exists Y \subseteq X, |Y| \leq \nu_\phi(|M|), \model \models_Y \phi.
\end{equation*}
\end{theorem}
\begin{proof}
By Proposition \ref{propo:dnf}, we can assume that $\phi$ is of the form $\bigsqcup_{i=1}^n \psi_i$, where all $\psi_i$ are in $\FO(=\!\!(\cdot), \DD)$. Furthermore, by Lemma \ref{lemma:constout} we can assume that every $\psi_i$ is of the form $\exists \tuple w_i (=\!\!(\tuple w_i) \wedge \theta_i)$, for $\theta_i \in \FO(\DD)$ and all tuples of variables $\tuple w_i$ are new. Now suppose that $\model \models_X \phi$: then there exists an $i \in 1 \ldots n$ and a tuple of elements $\tuple m \in M$ such that $\model \models_{X[\tuple m / \tuple w_i]} \theta_i$. But then there exists a $Y \subseteq X[\tuple m / \tuple w_i]$, $|Y| \leq \nu_{\theta_i}(|M|)$, such that $\model \models_Y \phi$. Now let $Z$ be the restriction of $Y$ to the domain of $X$: clearly $Z \subseteq X$ and $|Z| \leq |Y| \leq \nu_{\theta_i}(|M|) \leq \nu_{\phi}(|M|)$, and furthermore $\model \models_Z \exists \tuple w_i (=\!\!(\tuple w_i) \wedge \theta_i)$ and so in conclusion $\model \models_Z \phi$.
\end{proof}

\begin{theorem}
Let $k' > k$, and let $\DD$ be a class of $k$-ary upwards-closed (not necessarily first-order) dependencies. Then $\all_{k'}$ is not definable in $\FO(=\!\!(\cdot), \DD, \sqcup)$.
\end{theorem}
\begin{proof}
Suppose that $\phi(\tuple v) \in \FO(=\!\!(\cdot), \DD, \sqcup)$ defines $\all_{k'}$. Then, since all dependencies in $\DD$ are $|M|^k$-bounded, we have at once that $\phi$ is $q|M|^{k}$-bounded for some $q \in \mathbb N$. Now let $n \in \mathbb N$ be such that $n^{k'} > q n^k$, let $M$ be a model in the empty signature with $n$ elements, let $\tuple v$ be a tuple of $k'$ variables, and let $X = \{\emptyset\}[M/\tuple v]$. Then $M \models_X \all_{k'} \tuple v$, and therefore $M \models_X \phi(\tuple v)$. But then there must be a $Y \subseteq X$, $|Y| \leq q n^k$, such that $M \models_Y \phi(\tuple v)$; and this is not possible, because for such a $Y$ we would have that $M \not \models_Y \all_{k'} \tuple x$.
\end{proof}
In particular, it follows at once from this that $\all_{k+1}$ is not definable in $\FO(=\!\!(\cdot), \all_k, \sqcup)$. On the other hand if $k' < k$ the operator $\all_{k'} \tuple v$ is easily seen to be definable as $\forall \tuple w (\all_k \tuple v \tuple w)$; therefore 
\begin{corollary}
For all $k \in \mathbb N$, $\FO(=\!\!(\cdot), \all_k, \sqcup) \subsetneq \FO(=\!\!(\cdot), \all_{k+1}, \sqcup)$ (and all these logics are equivalent to first-order logic over sentences).
\end{corollary}
\section{$0$-ary Dependencies: Escaping the Empty Team}
\label{sec:zerary}
As a limit case of the notion of dependency, we have that
\begin{definition}
A $0$-ary dependency $\mathbf D$ is a set of models over the empty signature. For all models $\model$ and teams $X$, $\model \models_X \mathbf D$ if and only if $M \in \mathbf D$. 
\end{definition}
If a $0$-ary dependency is first-order, we have that $\model \models_X \mathbf D$ if and only if $M \models \mathbf D^*$, where $\mathbf D^*$ is a sentence over the empty signature; therefore, it is natural to generalize them all to an operator $[\cdot]$ of the form
\begin{description}
\item[\textbf{TS-$[\cdot]$:}] For all first-order sentences $\phi$ in the signature of $\model$, $\model \models_X [\phi]$ if and only if $\model \models \phi$ according to the usual Tarski semantics. 
\end{description}

Whenever $X$ is nonempty it follows at once from Proposition \ref{propo:flat} that $\model \models_X [\phi]$ if and only if $\model \models_X \phi$; but since $\model \models_\emptyset \phi$ for all first-order sentences $\phi$, in first-order logic with team semantics we have no way of verifying whether a given first-order sentence is true of our model when we are considering satisfiability with respect to the empty team. Therefore, we will add this $[\cdot]$ operator to our language. It is easy to see that adding it to a strongly first-order extension of first-order logic does not break the property of being strongly first-order:
\begin{proposition}
Let $\DD$ be any family of dependencies, and let $\phi \in \FO(\DD, [\cdot])$. Then $\phi$ is logically equivalent to some sentence of the form $\bigwedge_i[\theta_i] \wedge \psi$, where $\psi\in \FO(\DD)$. 
\end{proposition}
\begin{proof}
The proof is by induction on $\phi$, and it is entirely straightforward. We report only the case of disjunction:
\begin{itemize}
\item For all first-order sentences $\theta_i$, $\theta'_j$ and all $\FO(\DD)$ formulas $\psi_1,\psi_2$ we have that $(\bigwedge_i [\theta_i] \wedge \psi_1) \vee (\bigwedge_j [\theta'_j] \wedge \psi_2)$ is logically equivalent to $\bigwedge_i [\theta_i] \wedge \bigwedge_j [\theta'_j] \wedge (\psi_1 \vee \psi_2)$. Indeed, suppose that $X = Y \cup Z$ for two $Y$, $Z$ such that $\model \models_Y \bigwedge_i [\theta_i] \wedge \psi_1$ and $\model \models_Z \bigwedge_j [\theta'_j] \wedge \psi_2$. Then $\model \models \bigwedge_i \theta_i \wedge \bigwedge_j \theta'_j$, and therefore $\model \models_X \bigwedge_i [\theta_i] \wedge \bigwedge_j [\theta'_j]$; and since $\model \models_Y \psi$ and $\model \models_Z \theta$, we also have that $\model \models_{X} \psi \vee \theta$, and so in conclusion $\model \models_X \bigwedge_i [\theta_i] \wedge \bigwedge_j [\theta'_j] \wedge (\psi_1 \vee \psi_2)$.

The other direction is similar: if $\model \models \bigwedge_i \theta_i \wedge \bigwedge_j \theta'_j$ and $\model \models_X \psi_1 \vee \psi_2$ then $X = Y \cup Z$ for two $Y$ and $Z$ such that $\model \models_Y \psi_1$ and $\model \models_Z \psi_2$. But then $\model \models_Y \bigwedge_i [\theta_i] \wedge \psi_1$ and $\model \models_Z \bigwedge_j [\theta_j] \wedge \psi_2$, and so in conclusion $\model \models_X (\bigwedge_i [\theta_i] \wedge \psi_1) \vee (\bigwedge_j [\theta'_j] \wedge \psi_2)$.
\end{itemize}
\end{proof}
Therefore we have the following result: 
\begin{proposition}
\label{propo:zerary}
Let $\DD$ be a strongly first-order class of dependencies and let $\phi \in \FO(\DD, [\cdot])$ be a sentence. Then $\phi$ is logically equivalent to some first-order sentence $\phi'$, in the sense that $\model \models_{\startteam} \phi$ if and only if $\model \models \phi'$.
\end{proposition}
\begin{proof}
We may assume that $\phi$ is on the form $\bigwedge_i [\theta_i] \wedge \psi$, where $\psi$ is a $\FO(\DD)$-sentence. Now since $\DD$ is strongly first-order, $\psi$ is equivalent to some first-order $\psi'$; and since $\startteam$ is nonempty, we can take $\phi' = \bigwedge_i \theta_i \wedge \psi$. 
\end{proof}
\section{Unary Dependencies}
\label{sec:unary}
We will now consider the case of \emph{unary} dependencies, that is, of dependence atoms of arity one. As we will see, \emph{all} first-order unary dependencies are strongly first-order and definable in $\FO(=\!\!(\cdot), [\cdot], \all_1, \sqcup)$.

In order to prove this we will make use of the following standard result:
\begin{lemma}
Let $\phi$ be a first-order sentence over the vocabulary $\{P\}$, where $P$ is unary. Then $\phi$ is logically equivalent to a Boolean combination of sentences of the form $\exists^{=k} x P x$ and $\exists^{=k} x \lnot P x$.
\end{lemma}

Therefore, in order to show that all unary dependencies are in $\FO(=\!\!(\cdot), [\cdot], \all_1, \sqcup)$ it suffices to show that the following four dependencies are in it: 
\begin{description}
\item[\textbf{TS-eq-pos}:] For all $k \in \mathbb N$, $\model \models_X |v| = k$ iff $|X(v)| = k$; 
\item[\textbf{TS-neq-pos}:] For all $k \in \mathbb N$, $\model \models_X |v| \not = k$ iff $|X(v)| \not = k$; 
\item[\textbf{TS-eq-neg}:] For all $k \in \mathbb N$, $\model \models_X |M - v| = k$ iff $|M \backslash X(v)| = k$; 
\item[\textbf{TS-neq-neg}:] For all $k \in \mathbb N$, $\model \models_X |M- v| \not = k$ iff $|M \backslash X(v)| \not = k$.
\end{description}

Let us prove that this is the case. 
\begin{lemma}
The nonemptiness atom $\nonempty$ is definable in $\FO(\all_1)$ as $\forall q \all_1 q$.
\end{lemma}
\begin{proof}
Suppose that $\model \models_X \nonempty$, that is, $X \not = \emptyset$, and let $s \in X$. Then for all $m \in M$, $s[m/q] \in X[M/v]$, and thus $X[M/q](q) = M$, and thus $\model \models_X \forall q \all_1 q$ as required. 

However, if $X = \emptyset$ we have that $X[M/q] = \emptyset$ too, and thus $X[M/q](q) = \emptyset \not = M$, and finally $\model \not \models_X \forall q \all_1 q$.
\end{proof}
\begin{definition}
For all $k \in \mathbb N$ and all variables $v$, we define the following formulas: 
\begin{align*}
&\phi_{\leq k}(v) = \exists p_1 \ldots p_k (\bigwedge_{i} =\!\!(p_i) \wedge \bigvee_{i=1}^k v = p_i);\\
&\phi_{\geq k}(v) = \exists p_1 \ldots p_k (\bigwedge_i =\!\!(p_i) \wedge \bigwedge_{i \not = j} p_i \not = p_j \wedge \bigwedge_i (\nonempty \upharpoonright v = p_i));\\
&\psi_{\leq k}(v) = [\exists^{\leq k} x (x = x)] \sqcup \exists p_1 \ldots p_k(\bigwedge_i =\!\!(p_i) \wedge \exists q (\all_1(q) \wedge (\bigvee_i q = p_i \vee q = v);\\
&\psi_{\geq k}(v) = (\bot \wedge [\exists^{\geq k} x (x = x)]) \sqcup (\nonempty \wedge \exists p_1 \ldots p_k (\bigwedge_i =\!\!(p_1) \wedge \bigwedge_{i \not = j} p_i \not = p_j \wedge \bigwedge_{i=1}^k v \not = p_i))
\end{align*}
\end{definition}
\begin{proposition}
For all $k \in \mathbb N$, all variables $v$, all models $\model$ and all nonempty teams $X$ whose domain contains $v$, 
\begin{itemize}
\item $\model \models_X \phi_{\leq k}(v)$ if and only if $|X(v)| \leq k$; 
\item $\model \models_X \phi_{\geq k}(v)$ if and only if $|X(v)| \geq k$; 
\item $\model \models_X \psi_{\leq k}(v)$ if and only if $|M \backslash X(v)| \leq k$; 
\item $\model \models_X \psi_{\geq k}(v)$ if and only if $|M \backslash X(v)| \geq k$.
\end{itemize}
\end{proposition}
\begin{proof}
~ 
\begin{itemize}
\item Suppose that $\model \models_X \phi_{\leq k}(v)$ and $X$ is nonempty: then there exist elements $m_1 \ldots m_k$ such that for $Y = X[m_1 \ldots m_k / p_1 \ldots p_k]$, $\model \models_Y \bigvee_{i=1}^k v = p_i$. But then $X(v) \subseteq \{m_1 \ldots m_k\}$, and thus $|X(v)| \leq k$. If instead $X$ is empty then trivially $|X(v)| = 0 \leq k$.

Conversely, suppose that $X(v) = \{m_1, \ldots m_{k'}\}$ for $k' \leq k$. Then for\\ $Y = X[m_1 \ldots m_{k'} \ldots m_{k'} / p_1 \ldots p_k]$ we have that $\model \models_Y \bigvee_{i=1}^k v = p_i$. Thus $\model \models_X \phi_{\leq k}(v)$, as required.
\item Suppose that $\model \models_X \phi_{\geq k}(v)$. Then there exist distinct elements $m_1 \ldots m_k$ such that for $Y = X[m_1 \ldots m_k/p_1 \ldots p_k]$ and for all $i \in 1 \ldots k$, $\model \models_Y \nonempty \upharpoonright v = p_i$. Thus for all such $i$ there exists a $s \in Y$ with $s(v) = s(p_i) = m_i$, and thus $|X(v)| = |Y(v)| \geq k$. 

Conversely, suppose that $\{m_1 \ldots m_k\} \subseteq X(v)$, where all $m_i$ are distinct. Now take $Y = X[m_1 \ldots m_k / p_1 \ldots p_k]$: clearly $\model \models_Y \bigwedge_i =\!\!(p_i)\wedge \bigwedge_{i \not = j} p_i \not = p_j$, and it remains to show that for all $i$ $\model \models_Y \nonempty \upharpoonright v = p_i$. But $Y \upharpoonright (v = p_i) = \{s \in Y : s(v) = s(p_i) = m_i\}$ is nonempty by hypothesis, and this concludes the proof.
\item Suppose that $\model \models_X \psi_{\leq k} (v)$.  If $\model \models_X [\exists^{\leq k} x (x = x)]$ we have that $|M| \leq k$, from which it follows at once that $|M \backslash X(v)| \leq |M| \leq k$. Otherwise, we can find elements $m_1 \ldots m_k$ such that, for $Y = X[m_1 \ldots m_k / p_1 \ldots p_k]$, there exists a choice function $F$ for which $\model \models_{Y[F/q]} \all_1(q) \wedge (\bigvee_i q = p_i \vee q = v)$. Then $M \backslash X(v)$ must be contained in $\{m_1 \ldots m_k\}$, since $q$ takes all possible values and $s(q) \not \in \{m_1 \ldots m_k\} \Rightarrow s(q) = s(v)$.

Conversely, suppose that $M \backslash X(v) \subseteq \{m_1 \ldots m_{k}\}$. If $X \not = \emptyset$, let $Y$ be\\$X[m_1 \ldots m_{k} / p_1 \ldots p_k]$, and for all $s \in Y$ let $F(s) = \{m_1 \ldots m_{k}\} \cup \{s(v)\}$. Then $Y[F/q] \models \all_1 q$: indeed, if $m \in \{m_1 \ldots m_{k}\}$ then $m \in F(s)$ for all $s \in Y$, and otherwise $m = s(x)$ for some $s \in Y$ (and hence $m \in F(s)$ for this choice of $s$). Furthermore, for all $h \in Y[F/q]$, if $h(q) \not \in \{m_1 \ldots m_k\}$ then we have that $h(q) = h(v)$, as required. If instead $X = \emptyset$ then $|M| = |M \backslash X(v)| \leq k$, and hence $\model \models_X \exists^{\leq k} x (x = x)$. 
\item Suppose that $\model \models_X \psi_{\geq k}(v)$ and $X \not = \emptyset$. Then there exist distinct elements $m_1 \ldots m_k$ such that for $Y = X[m_1 \ldots m_k / p_1 \ldots p_k]$, $\model \models_Y \bigwedge_{i=1}^k v \not = p_i$. Therefore $\{m_1 \ldots m_k\} \in M \backslash X$, and thus $|M \backslash X| \geq k$. If instead $X = \emptyset$ then $\model \models \bot \wedge [\exists^{\geq k} x (x = x)]$ and hence $|M| = |M \backslash X(v)| \geq k$ as required. 

Conversely, suppose that $|M \backslash X(v)| \geq k$. If $X$ is nonempty we can choose elements $m_1 \ldots m_k \in M \backslash X(v)$ and verify that $\model \models_{X[m_1 \ldots m_k / p_1 \ldots p_k]} \bigwedge_{i \not = j} p_i \not = p_j \wedge \bigwedge_{i} v \not = p_i$; and if $X$ is empty then it follows at once that $|M| \geq k$ and hence that $\model \models_X \bot \wedge [\exists^{\geq k} x (x = x)]$, as required.
\end{itemize}
\end{proof}
\begin{corollary}
For all $k \in \mathbb N$, the atoms $|v| = k$, $|v| \not = k$, $|M - v| = k$ and $|M - v| \not = k$ are all definable in $\FO(=\!\!(\cdot), \all_1, \sqcup)$. 
\end{corollary}
\begin{proof}
Observe that 
\begin{itemize}
\item $\model \models_X |v| = k$ iff $\model \models_X \phi_{\leq k} (v) \wedge \phi_{\geq k} \phi$; 
\item $\model \models_X |v| \not = k$ iff $\model \models_X \phi_{\leq k-1}(v) \sqcup \phi_{\geq k+1}(v)$; 
\item $\model \models_X |M - v| = k$ iff $\model \models_X \psi_{\leq k}(v) \wedge \psi_{\geq k}(v)$; 
\item $\model \models_X |M - v| \not = k$ iff $\model \models_X \psi_{\leq k-1}(v) \sqcup \psi_{\geq k+1}(v)$
\end{itemize}
where we let $\phi_{\leq -1} = \psi_{\leq -1} = \bot$.
\end{proof}
Putting everything together, we have that 
\begin{theorem}
Every unary first-order dependency is definable in $\FO(=\!\!(\cdot), \all_1, \sqcup)$.
\end{theorem}
\begin{proof}
Let $\Dep$ be a unary first-order dependency and let $v$ be a first-order variable. By definition, $\model \models_X \Dep v$ if and only if $(M, X(v)) \models \Dep^*(P)$, where $\Dep^*(P)$ is a first-order formula in the vocabulary $\{P\}$ ($P$ unary). But then $\Dep^*(P)$ is equivalent to a Boolean combination of sentences of the form $\exists^{=k} x Px$ and $\exists^{=k} x \lnot Pk$; and thus, we may assume that $\Dep^*(P)$ is of the form $\bigvee_i \bigwedge_j \theta_{ij}$, where each $\theta_{ij}$ is $\exists^{=k} x Px$, $\exists^{=k} x \lnot Px$, or a negation of a formula of this kind. But then $\Dep v$ is logically equivalent to 
\begin{equation*}
\bigsqcup_i \bigwedge_j \theta'_{ij}, 
\end{equation*}
where 
\begin{itemize}
\item If $\theta_{ij}$ is $\exists^{=k} x Px$, $\theta'_{ij}$ is $|v| = k$; 
\item If $\theta_{ij}$ is $\lnot \exists^{=k} x Px$, $\theta'_{ij}$ is $|v| \not = k$; 
\item If $\theta_{ij}$ is $\exists^{=k} x \lnot Px$, $\theta'_{ij}$ is $|M - v| = k$; 
\item If $\theta_{ij}$ is $\lnot \exists^{=k} x \lnot Px$, $\theta'_{ij}$ is $|M - v| \not = k$.
\end{itemize}
\end{proof}
Finally, we need to show that every sentence of $\FO(=\!\!(\cdot), [\cdot], \all_1, \sqcup)$ is equivalent to some first-order sentence. But this is straightforward: 
\begin{theorem}
Let $\phi \in \FO(=\!\!(\cdot), \all_1, \sqcup, [\cdot])$ be a sentence. Then $\phi$ is logically equivalent to some first-order sentence.
\end{theorem}
\begin{proof}
By Proposition \ref{propo:dnf}, $\phi$ is equivalent to some sentence of the form $\sqcup_i \psi_i$, for $\psi_i \in \FO(=\!\!(\cdot), \all_1, [\cdot])$. Observe further that all expressions $[\theta]$ which occur in our formulas are such that $\theta$ is a first-order sentence over the empty vocabulary; and therefore, these expressions are trivially upwards-closed first-order dependencies, since for any fixed model they either hold in all teams or in none of them.\footnote{On the other hand, if $\theta$ were a first-order sentence over the non-empty vocabulary then it would not be a dependency.} Then by Theorem \ref{thm:uwc} and Proposition \ref{propo:zerary} every such sentence is equivalent to some first-order sentence $\psi'_i$ and thus $\phi$ is equivalent to $\bigvee_i \psi'_i$.
\end{proof}
Putting everything together, we have that 
\begin{corollary}
Let $\Dep$ be a unary first-order dependency. Then it is strongly first-order and definable in $\FO(=\!\!(\cdot), [\cdot], \all_1, \sqcup)$.
\end{corollary}

We conclude this section by mentioning an open problem.
\runinhead{Question:} Let $k > 1$. Are there any strongly first-order $k$-ary dependencies which are not definable in $\FO(=\!\!(\cdot), [\cdot], \all_k, \sqcup)$? 
\section{Conclusion}
Much of the team semantics research has so far focused on formalisms which are greatly more expressive than first-order logic. However, the study of weaker extensions of first-order logic, which do not rise above it insofar as the definability of classes of models is concerned, promises to be also of significant value: not only this investigation offers an opportunity of examining the nature of the boundary between first- and second-order logic, but it also provides us with (comparatively) computationally ``safe'' classes of dependencies and operators to use in applications.

This work builds on the results of \cite{galliani13e} and can only be an initial attempt of making sense of the wealth of these ``weak'' extensions of first-order logic with team semantics. Much of course remains to be done; but a few distinctive characteristics of this line of investigation may be gleaned already.
\begin{itemize}
\item The totality atoms $\all_k$ seem to have a role of particular relevance in the theory of strongly first-order dependencies. It remains to be seen whether this role will be preserved by the further developments of the theory; but in any case, the fact that these atoms are the ``maximally unbounded'' (in the sense of Definition \ref{defin:bounded}) ones for their arities is certainly suggestive, as is the existence of a strict definability hierarchy based on their arities and the fact that all monadic first-order dependencies are definable in terms of the $\all_1$ atom.
\item The logic $\FO(\sim) = \FO(\nonempty, \sqcup)$, as the simplest extension of first-order logic with team semantics which is closed under contradictory negation, is also an item of particular interest. As we saw, it suffices to add to it comparatively harmless dependencies such as constancy atoms to obtain the full expressive power of second-order logic; thus, despite its simplicity, this logics appears to be a natural ``stopping point'' in the family of dependency-based extensions of first-order logic, deserving of a more in-depth study of its properties.
\item When working with classes of strongly first-order dependencies, different choices of connectives and operators emerge to the foreground. In particular, the role of the classical disjunction $\phi \sqcup \psi$ in the study of dependence logic and its extensions has been relatively marginal so far; but nonetheless, this connective proved itself of fundamental importance for many of the results of this work. More in general, it appears now that a fully satisfactory account of dependencies and definability cannot be developed if not by integrating it with a general theory of \emph{operators} and \emph{uniform definability} in team semantics. The work of \cite{engstrom12,engstrom12b,kuusisto13} on generalized quantifiers in team semantics seems to be the most natural starting point for such an enterprise; in particular, it would be worthwhile to be able to characterize general families of dependencies \emph{and operators} which do not increase the expressive power of first-order logic (wrt sentences). 
\end{itemize}
\begin{acknowledgement}
This research was supported by the Deutsche Forschungsgemeinschaft (project number DI 561/6-1).
\end{acknowledgement}
\bibliographystyle{spmpsci}
\bibliography{biblio}
\end{document}